\documentclass[11pt]{amsart}
\usepackage{amsthm}
\usepackage{amsmath}
\usepackage{amssymb}
\usepackage{amsfonts}
\usepackage{amscd}
\usepackage[mathscr]{eucal}
\usepackage{verbatim}
\usepackage{latexsym}
\usepackage{graphics}
\usepackage[all]{xy}
\usepackage{dev}

\theoremstyle{plain}
\newtheorem{thm}{Theorem}

\newtheorem{defn}[thm]{Definition}

\newtheorem{lem}[thm]{Lemma}
\newtheorem{propn}[thm]{Proposition}

\def\A{{\mathcal A}}
\def\C{{\mathfrak C}}
\def\G{{\Gamma}}
\def\H{{\mathcal H}}
\def\g{{\gamma}}

\def\Aut{{\rm Aut}}
\def\GL{{\rm GL}}

\def\SO{{\rm SO}}
\def\Spin{{\rm Spin}}
\def\ot{{\otimes}}
\def\tr{{\rm tr}}

\begin{document}

\title{Excellence of $G_2$ and $F_4$}
\author{Shripad M. Garge}
\address{Shripad M. Garge. Department of Mathematics, Indian Institute of Technology Bombay, Powai, Mumbai. 400 076. INDIA.}
\email{shripad@math.iitb.ac.in}

\begin{abstract}
A linear algebraic group $G$ defined over a field $k$ is said to be excellent if for every field extension $L$ of $k$ the anisotropic kernel of the group $G \ot_k L$ is defined over $k$. 
We prove that groups of type $G_2$ and $F_4$ are excellent over any field $k$ of characteristic other than $2$ and $3$.
\end{abstract}

\maketitle


\begin{centerline}
{\small {\dn k\4\314w sO\314w aAIQyA -\9{m}(yT\0}}
\end{centerline}
\vskip5mm

One basic (and difficult) problem in the theory of linear algebraic groups is to understand the structure of anisotropic groups. 
With this motivation, Kersten and Rehmann introduced the notion of excellent algebraic groups in \cite{KR} in analogy with the notion of excellent quadratic forms introduced by Knebusch (\cite{K1}). 
The excellence properties of some classical groups have been studied in \cite{KR, IK}.
In this note, we study the excellence properties of the groups of type $G_2$ and $F_4$. 
In contrast with the cases of classical groups, the groups $G_2$ and $F_4$ are always excellent over a field of characteristic different from $2$ and $3$. 
The main theorems of the paper are the following ones. 

\begin{thm}\label{thm:1}
Let $k$ be a field of characteristic different from $2$ and $3$, and let $G$ be a group of type $G_2$ defined over $k$. 
Then $G$ is excellent over $k$.
\end{thm}

\begin{thm}\label{thm:main}
Let $k$ be a field of characteristic different from $2$ and $3$, and let $G$ be a group of type $F_4$ defined over $k$. 
Then $G$ is excellent over $k$.
\end{thm}

We now describe the contents of this paper.
In $\S 1$ we give some basic definitions and prove Theorem \ref{thm:1}.
A group of type $F_4$ is described as the automorphism group of an Albert algebra. 
We recall the definition of an Albert algebra and some basic facts pertaining to it in $\S 2$. 
To study the excellence properties of groups, we need to compute their anisotropic kernels.
We calculate the anisotropic kernels of groups of type $F_4$ in $\S 3$. 
Finally, in $\S 4$, we give a proof of Theorem \ref{thm:main}. 

At the end of $\S 4$, we give another proof of Theorem \ref{thm:main}, due to an anonymous referee. 
However, in our proof, we have explicitly calculated the anisotropic kernel. 
We hope that this description will be useful for other purposes. 
In particular, since a group of type $E_6$ is also described using Albert algebras, the explicit description of the anisotropic kernel and, in general, the methods of this note might be useful in handling the $E_6$ case. 

The author thanks Maneesh Thakur and Dipendra Prasad for several fruitful discussions.

\section{Basic definitions and proof of theorem \ref{thm:1}}
Let $k$ be a field and let $G$ be a connected semisimple linear algebraic group defined over $k$. 
Choose a maximal $k$-split torus $S$ in $G$ and let $Z(S)$ denote the centraliser of the torus $S$ in $G$. 
The connected component of the derived group $DZ(S) = [Z(S), Z(S)]$ is defined to be the {\em anisotropic kernel} of $G$ over $k$. 
Since any two maximal $k$-split tori in $G$ are conjugate in $G(k)$, it follows that, up to isomorphism, the anisotropic kernel of $G$ over $k$ does not depend on the choice of $S$. 
We denote it by $G_{an}$. 

By construction, $G_{an}$ is a connected anisotropic semisimple $k$-subgroup of $G$. 
One sees that the group $G$ is anisotropic if and only if it equals $G_{an}$ and that the group $G$ is quasi-split if and only if $G_{an}$ is trivial. 
The anisotropic kernel $G_{an}$ is an important ingredient of the data, together with the index of $G$ and the type of $G$, that determines the $k$-isomorphism class of $G$ (\cite[\S 2.7]{T}).

This notion is analogous to the notion of anisotropic kernel of a regular quadratic form defined over a field. 
Indeed the result stated in the previous paragraph is also analogous to Witt's theorem that a regular quadratic form defined over $k$ is determined up to $k$-isomorphism by its anisotropic kernel and its Witt index. 

Knebusch (\cite{K1, K2}) has developed the notion of excellent quadratic forms which has been useful in the study of regular quadratic forms. 
Kersten and Rehmann (\cite{KR}) introduced the notion of excellent algebraic groups in an analogous way. 

\begin{defn}
Let $G$ be a linear algebraic group defined over a field $k$. 
The group $G$ is called excellent over $k$ if for any field extension $L$ of $k$ there exists a group $H$ defined over $k$ such that $H \otimes_k L$ is isomorphic to the anisotropic kernel of $G \otimes_k L$. 
\end{defn}

We remark here that the $k$-group $H$, if it exists, need not be unique (up to $k$-isomorphism) group with the property described in the above definition. 
Indeed, any $L/k$-form of the anisotropic kernel of $G \otimes_k L$, if there exists one, will satisfy the required condition.

Excellence properties of some groups of classical type are studied in \cite{KR} and later in \cite{IK}. 
We now prove Theorem \ref{thm:1} that a group of type $G_2$ is excellent over a field of characteristic different from $2$ and $3$. 

\begin{proof}[Proof of Theorem \ref{thm:1}.] 
It is known that a group $G$ of type $G_2$ is either anisotropic or it is split (\cite[Proposition 17.4.2]{Sp}). 
Therefore the anisotropic kernel $G_{an}$ over any extension is either the whole group or the trivial subgroup, and hence it is always defined over the base field $k$.
\end{proof}

\section{Albert algebras and groups of type $F_4$.}\label{section:definitions}

We recall some basic properties of Albert algebras and groups of type $F_4$ in this section. 
Our main reference is \cite[Chapter 5]{SV}. 

Throughout this section, $k$ denotes a field of characteristic other than $2$ and $3$. 
Sometimes we state results which hold for fields without any restriction on the characteristic (Theorem \ref{thm:albert}, for instance).
Then we use the letter $K$ to denote a field. 

Let $\C$ be an octonion algebra defined over $k$ and let $c \mapsto \overline{c}$ denote the conjugation in $\C$.
Fix a diagonal matrix $\G \in \GL_3(k)$ with entries $\g_1, \g_2, \g_3$. 
Consider the algebra of the hermitian elements with respect to $\G$
$$\H(\C; \G) = \H(\C; \g_1, \g_2, \g_3) := \left\{
\begin{pmatrix}
x_1 & c_3 & {\g_1}^{-1} \g_3 \overline{c_2} \\
{\g_2}^{-1} \g_1 \overline{c_3} & x_2 & c_1 \\
c_2 & {\g_3}^{-1} \g_2 \overline{c_1} & x_3 
\end{pmatrix} :x_i \in k, c_i \in \C\right\} .$$
We define a new multiplication in $\H(\C; \G)$ by $x y := (x \cdot y + y \cdot x)/2$, where the product $x \cdot y$ denotes the usual matrix multiplication, making $\H(\C; \G)$ a non-associative $k$-algebra. 
In fact, $\H(\C;\G)$ is what is called a non-associative {\em exceptional Jordan algebra}.

\begin{defn}
An Albert algebra over $k$ is a $k$-algebra $\A$ such that for an extension $L/k$, $\A \ot_k L$ is isomorphic to some $\H(\C; \g_1, \g_2, \g_3)$ defined over $L$. 
\end{defn}

The product in $\H(\C; \G)$ gives rise to a quadratic norm on $\H(\C; \G)$, $Q(x) = \frac{1}{2} \tr(x^2)$, and the associated bilinear form is given by 
$$\langle x, y \rangle := Q(x + y) - Q(x) - Q(y) = \tr(xy) .$$
Further, the octonion algebra $\C$ is determined up to isomorphism by the Albert algebra $\H(\C; \G)$ (\cite[\S 5.6]{SV}) and we call it {\em the coordinate algebra of} $\H(\C, \G)$.

Albert algebras are important because they give explicit descriptions of the exceptional groups of type $F_4$ and $E_6$. 

\begin{thm}[{\cite[page 164]{H}}, {\cite[17.6.9]{Sp}}]\label{thm:albert}
Let $K$ be a field and let $G$ be a group of type $F_4$ defined over $K$. 
Then there exists an Albert algebra $\A$ $($unique up to isomorphism$)$ defined over $K$ such that $G$ is isomorphic to the algebraic $K$-group $\Aut_K(\A)$. 
\end{thm}

It is known that the possible split ranks of a group $G$ of type $F_4$ are $0, 1$ and $4$ (\cite[page 60]{T}, \cite[Chapter 17]{Sp}). 
We can also read off the split rank of $G$ from the structure of the corresponding Albert algebra.

If $G$ is anisotropic, i.e, if it has split rank $0$ then it is $k$-isomorphic to $\Aut_k(\A)$ where the Albert algebra $\A$ has no nontrivial nilpotent elements. 
If $G$ is split then the corresponding Albert algebra is $k$-isomorphic to some $\H(\C; \G)$ where $\C$ is the split octonion algebra over $k$. 
Finally, if $G$ has split rank $1$ then the corresponding Albert algebra $\A$ has nontrivial nilpotent elements but does not have two non-proportional orthogonal ones.
The structure of such an Albert algebra is given explicitly by the following result, which is essentially due to Albert and Jacobson (\cite[Theorem 6]{AJ}). 

\begin{propn}\label{propn:rank1}
Let $k$ be a field of characteristic different from $2$ and $3$, and let $\A$ be an Albert algebra defined over $k$. 
If the algebra $\A$ has nontrivial nilpotent elements but does not have two non-proportional orthogonal nilpotent elements then it is $k$-isomorphic to $\H(\C; 1, -1, 1)$ where $\C$ is a division octonion algebra defined over $k$.
\end{propn}

\begin{proof}
Indeed, if the Albert algebra $\A$ contains a nontrivial nilpotent element, then it is reduced (\cite[Theorem 5.5.1]{SV}).
Hence the algebra $\A$ is $k$-isomorphic to $\H(\C;\G)$ for some octonion algebra $\C$ and for some diagonal matrix $\G \in \GL_3(k)$ (\cite[Theorem 5.4.5]{SV}).
If the octonion algebra $\C$ is split, then $\A$ itself is split and then it is clear that it contains two non-proportional orthogonal nilpotent elements. 
Since we assume that it is not the case for $\A$, the octonion algebra $\C$ is a division algebra. 
Moreover, by \cite[Theorem 6]{AJ} the matrix $\G$ can be taken to be the diagonal matrix with entries $1, -1, 1$.
\end{proof}

We will need more results regarding the reduced Albert algebras, i.e., the Albert algebras that contain nontrivial nilpotent elements. 

Fix a reduced Albert algebra $\A$ over $k$ and a nontrivial idempotent $u \in \A$. 
Assume that $Q(u) = 1/2$, where $Q$ denotes the norm form on $\A$. 
Such an idempotent always exists and is called a primitive idempotent, as it can not be written as $u_1 + u_2$ where $u_i^2 = u_i$ and $u_1u_2 = 0$ (\cite[Lemma 5.2.2]{SV}). 
We describe certain subgroups in the group $\Aut_k(\A)$ in terms of this idempotent. 
The bilinear form associated to $Q$ is denoted by $\langle \cdot, \cdot \rangle$. 
Define 
$$E_0 := \big\{x \in \A: \langle x, 1\rangle = \langle x, u\rangle = 0, ux = 0\big\} .$$
The restriction of $Q$ to $E_0$ is non-degenerate (\cite[\S 5.3]{SV}). 
Let $G_u$ be the subgroup of $\Aut_k(\A)$ consisting of the automorphisms of $\A$ which fix the idempotent $u$. 
If $Q_0$ denotes the restriction of the norm form $Q$ to $E_0$, we get the following description of the group $G_u$ (\cite[Proposition 7.1.6]{SV}).

\begin{propn}\label{propn:u}
The group $G_u$ is $k$-isomorphic to $\Spin(Q_0, E_0)$. 
\end{propn}

\section{Calculation of anisotropic kernel of a group of type $F_4$}
In this section, we calculate the anisotropic kernel of a group of type $F_4$. 
Let us fix a field $k$ of characteristic other than $2, 3$ and a connected group $G$ of type $F_4$ defined over $k$. 
It is known that the possible split ranks of $G$ are $0, 1$ or $4$ (\cite[page 60]{T}). 

If the split ranks of $G$ are $0$ or $4$, i.e., if the group $G$ is either $k$-anisotropic or if it is $k$-split, then the respective anisotropic kernels are the whole group $G$ and the trivial subgroup. 

So, now on we assume that the split rank of $G$ is $1$. 
Then it follows, by Proposition \ref{propn:rank1} and the uniqueness of the Albert algebra in Theorem \ref{thm:albert}, that $\A = \H(\C; 1, -1, 1)$ for an octonion algebra $\C$ defined over $k$. 
Let $Q$ denote the norm form on the Albert algebra $\A$ and let $\langle \cdot, \cdot \rangle$ be the corresponding bilinear form. 

To calculate the anisotropic kernel $G_{an}$, we need the description of a maximal $k$-split torus in $G$.
For that, we use the following result (\cite[Lemma 5.1]{PST}), which is clear from the fact that the elements of $\H(\C; \G)$ are precisely those elements which are {\em hermitian} with respect to $\G$. 

\begin{lem}\label{lem:torus}
Let $\A = \H(\C; \G)$ be an Albert algebra over a field $K$ where $\G$ is a diagonal matrix in $\GL_3(K)$. 
There exists a map $\phi: \SO(\G) \rightarrow \Aut_K(\A)$ which sends $X \in \SO(\G)$ to the automorphism of $\A$ given by $\theta \mapsto X \theta X^{-1}$. 
\end{lem}

This result gives us a map, which we continue to denote by $\phi$, from the group $\SO(1, -1, 1)$ to $G$. 
The group $\SO(1, -1, 1)$ is a split group. 
Indeed, it contains the following $1$-dimensional split torus, 
$$T = \left\{\begin{pmatrix}
a & b & 0 \\
b & a & 0 \\
0 & 0 & 1
\end{pmatrix}:a^2 -b^2 = 1\right\} .$$
Observe that the kernel of the map $\phi$ is finite, hence the image of the torus $T$ under the map $\phi$ is a $1$-dimensional split torus in $G$. 
We prove that this torus $\phi(T)$ can be embedded in a certain subgroup of $G$ of type $B_4$ and so the anisotropic kernel of $G \ot_k L$ can be described as the anisotropic kernel of the same subgroup. 

\begin{lem}\label{lem:b4}
There exists a primitive idempotent $u \in \A$ such that the anisotropic kernel of the group $G$ is the same as the anisotropic kernel of the subgroup $G_u$ of $G$, consisting of those automorphisms of $\A$ that fix $u$. 
\end{lem}

\begin{proof}
Define $u$ to be the idempotent
$$\begin{pmatrix} 0 & 0 & 0 \\ 0 & 0 & 0 \\ 0 & 0 & 1 \end{pmatrix} \in \A .$$
Since $Q(u) := \tr(u^2)/2 = 1/2$, $u$ is a primitive idempotent. 
Let $G_u$ denote the subgroup of $G$ consisting of the automorphisms of $\A$ which fix $u$.
By Proposition \ref{propn:u}, this subgroup $G_u$ is isomorphic to the group $\Spin(Q_0, E_0)$, where 
$$E_0 := \big\{x \in \A: \langle x, 1\rangle = \langle x, u\rangle = 0, ux = 0\big\} .$$
It can be easily seen that the space $E_0$ is given explicitly by:
\begin{eqnarray*}
E_0 & = & \left\{ \begin{pmatrix} 
x & c & 0 \\
- \overline{c} & -x & 0 \\
0 & 0 & 0
\end{pmatrix} \in \A \right\} \stackrel{\sim}{\longrightarrow} k \oplus \C .
\end{eqnarray*}
Thus, the space $E_0$ is a $9$-dimensional vector space over $k$ and hence the group $G_u = \Spin(Q_0, E_0)$ is a group of type $B_4$. 
The torus $\phi(T)$ fixes the idempotent $u$, therefore $\phi(T)$ is contained in the subgroup $G_u$ of $G$. 
It is a group of $k$-rank 1, as it contains the torus $\phi(T)$. 
By Tits classification (\cite[page 55, 56]{T}), we know that the anisotropic kernel of $G_u$ is a group of type $B_3$. 
Moreover, $G_{an}$ is also a group of type $B_3$ (\cite[page 61]{T}).
It is easy to see that $(G_u)_{an}$ is contained in $G_{an}$.
Since both these are connected groups, it is clear that $G_{an} = (G_u)_{an}$. 
\end{proof}

\section{Proof of theorem 2.}\label{section:main}
This section is devoted to the proof of Theorem 2. 
We fix a field $k$ of characteristic different from $2$ and $3$, and a group $G$ of type $F_4$ defined over $k$. 
Let $\A$ denote the corresponding Albert algebra defined over $k$ with $Q$ as the norm form and $\langle \cdot, \cdot \rangle$ as the corresponding bilinear form.  
Let us also fix a field extension $L$ of $k$. 
We need to show that the anisotropic kernel of the $L$-group $G \otimes_k L$ is defined over $k$. 

If $G \otimes_k L$ is split or anisotropic then the anisotropic kernel of $G \otimes_k L$, being respectively the trivial subgroup or the whole group $G \otimes_k L$, is defined over $k$. 
It is known that the only other possibility of the split rank of a group of type $F_4$ is 1 (\cite[page 60]{T}). 
So we now assume that the group $G$ is anisotropic and $G \otimes_k L$ has split rank 1. 
It then follows from Lemma \ref{lem:b4} that the anisotropic kernel $G_{an}$ is the same as the anisotropic kernel of a classical subgroup of $G$ of type $B_4$, $G_u$, for an explicit idempotent $u$ in the Albert algebra $\A \otimes_k L$. 

We now give a sufficient condition for the anisotropic kernel $(G_u)_{\rm an}$ to be defined over the field $k$. 
The condition is also necessary but we do not prove it since it is not required in our proof of Theorem \ref{thm:main}.

\begin{lem}\label{lem:kernel}
If there exists an octonion algebra $\C'$ defined over $k$ such that $\C' \ot_k L$ is isomorphic to the coordinate algebra of $\A \otimes_k L$ then the anisotropic kernel of the group $G_u$ is defined over the field $k$. 
\end{lem}

\begin{proof}
Let $\C$ denote the coordinate algebra of $\A \otimes_k L$. 
Observe that the form $Q_0$, the restriction of $Q \otimes_k L$ to $E_0$, is given by
$$Q_0: L \oplus \C \rightarrow L, \hskip2mm (x, c) \mapsto x^2 - N(c) ,$$
where $N$ is the norm form of the octonion algebra $\C$. 
If there exists an octonion algebra $\C'$ defined over $k$ such that $\C' \ot_k L$ is isomorphic to $\C$ and if $N'$ is the norm form of $\C'$, then $N' \ot_k L$ is equivalent to $N$. 
If we define a form $Q_0':k \oplus \C' \rightarrow k$, given by $(x, c') \mapsto x^2 - N'(c')$ then $Q_0'\ot_k L$ is equivalent to the form $Q_0$. 
Moreover, the form $Q_0'$ is isotropic over $k$, $Q_0'(1, 1) = 0$, and its Witt index over $k$ is not more than one as the Witt index over $L$ is one. 
It follows that the group $H = \Spin(Q_0', k \oplus \C')$ is a $k$-rank one group and $H_{\rm an} \ot_k L$ is isomorphic to the anisotropic kernel of the group $G_u$. 
\end{proof}

Now we state a nontrivial result, attributed to Serre and Rost (\cite[Theorem 1.8]{PR}), before going on to prove the main theorem. 

\begin{thm}\label{thm:Serre}
Let $J$ be an Albert algebra over a field $k$ of characteristic not $2$ or $3$. 
There exists a (unique) octonion algebra $\C$ defined over $k$ such that for any extension $L/k$ with the property that $J \ot_k L$ is reduced, $\C \ot_k L$ is the coordinate algebra of $J \ot_k L$. 
\end{thm}

\begin{proof}[Proof of Theorem 2]
It is clear from the discussion at the beginning of this section, that we need only to consider the case where $G$ is anisotropic over $k$ and the split rank of $G \otimes_k L$ is $1$.
By Lemma \ref{lem:b4}, the anisotropic kernel of the group $G \ot_k L$ is the same as that of the subgroup $G_u$. 
Therefore it is enough to prove that $(G_u)_{\rm an}$ is defined over $k$. 
By Lemma \ref{lem:kernel}, we only have to prove the existence of an octonion algebra $\C'$ over $k$ such that $\C' \ot_k L$ is isomorphic to the coordinate algebra $\C$ of $\A \otimes_k L$.
But this is precisely the statement of Theorem \ref{thm:Serre}. 
Thus, the anisotropic kernel of the group $G \ot_k L$ is defined over the field $k$ and hence the group $G$ is an excellent group. 
\end{proof}

We now give another proof of Theorem \ref{thm:main}, due to an anonymous referee. 

If the split rank of the group $G \otimes_k L$ is 1, then $G \otimes_k L$ is isomorphic to $\Aut_L(\A \otimes_k L)$ where $\A \otimes_k L$ is a reduced Albert algebra defined over $L$. 
Then there exists an octonion algebra $\C$ defined over $L$ such that $\A \otimes_k L$ is $L$-isomorphic to $\H(\C; 1, -1, 1)$. 
By Theorem \ref{thm:Serre}, the octonion algebra $\C$ is defined over $k$. 
Define $G'$ to be the $k$-automorphism group of the $k$-algebra $\H(\C; 1, -1, 1)$. 
Then $G'$ is a group of type $F_4$ and its $k$-rank is 1. 
Clearly, the anisotropic kernel of the group $G'$ is isomorphic to $(G \otimes_k L)_{an}$ over $L$ and hence the anisotropic kernel of the group $G \otimes_k L$ is defined over $k$. 

Indeed, the above theorem, Theorem \ref{thm:Serre}, says that an Albert algebra is excellent over any field $k$ of characteristic not $2$ or $3$ and hence it follows that a group of type $F_4$ is also excellent over $k$.


\begin{thebibliography}{9}

\bibitem[AJ]{AJ} {\bf Adrian A. Albert; Nathan Jacobson:} {\em On reduced exceptional simple Jordan algebras}, Annals of Math., {\bf 66}, (1957), 400--417.

\bibitem[Hi]{H} {\bf Hiroaki Hijikata:} {\em A remark on the groups of type $G_2$ and $F_4$}, J. Math. Soc. Japan, {\bf 15}, (1963), 159--164.

\bibitem[IK]{IK} {\bf Oleg H. Izhboldin; Ina Kersten:} {\em Excellent special orthogonal groups},  Doc. Math., {\bf 6} (2001), 385--412. (available electronically at {\tt http://www.mathematik.uni-bielefeld.de/documenta/vol-06/vol-06.html})

\bibitem[KR]{KR} {\bf Ina Kersten; Ulf Rehmann:} {\em Excellent algebraic groups I}, Journal of Algebra, {\bf 200}, (1998), 334--346. 

\bibitem[K1]{K1} {\bf Manfred Knebusch:} {\em Generic splitting of quadratic forms I}, Proc. London Math. Soc. (3), {\bf 33}, (1976), 65--93. 

\bibitem[K2]{K2} {\bf Manfred Knebusch:} {\em Generic splitting of quadratic forms II}, Proc. London Math. Soc. (3), {\bf 34}, (1977), 1--31. 

\bibitem[PST]{PST} {\bf R. Parimala; V. Suresh; Maneesh Thakur:} {\em Jordan algebras and $F_4$ bundles over the affine plane}, Journal of Algebra, {\bf 198}, (1997), 582--607.  

\bibitem[PR]{PR} {\bf Holger P. Petersson; Michel L. Racine:} {\em On the invariants mod $2$ of Albert algebras}, Journal of Algebra, {\bf 174}, (1995), 1049--1072. 

\bibitem[Sp]{Sp} {\bf T. A. Springer:} {\em Linear algebraic groups}, Progress in Mathematics, {\bf 9}, Birkhauser, (1998).

\bibitem[SV]{SV} {\bf T. A. Springer; F. D. Veldkamp:} {\em Octonions, Jordan algebras and exceptional groups}, Springer monographs in Mathematics, Springer-Verlag, (2000). 

\bibitem[Ti]{T} {\bf Jacques Tits:} {\em Classification of algebraic semisimple groups}, Algebraic Groups and Discontinuous Subgroups (Proc. Sympos. Pure Math., Boulder, Colo., 1965), Amer. Math. Soc., (1966), 33--62 

\end{thebibliography}
\end{document}